\documentclass[a4paper]{amsart}

\usepackage{amsfonts}
\usepackage{amssymb}
\usepackage{amstext}
\usepackage{graphicx}
\usepackage{tikz}
\usepackage{caption}

\usepackage[]{amscd}
\usepackage{hyperref}
\newtheorem{theorem}{Theorem}[section]
\newtheorem{lemma}{Lemma}[section]

\newtheorem{notation}{Notation}[section]

\title{ Transcendental lattices of certain singular $K3$ surfaces}
\usepackage{amsmath}
\begin{document}
\author{M.J.Bertin and  O.Lecacheux}
\curraddr{Sorbonne Universit\'e \\Institut de Math\'ematiques de Jussieu-Paris Rive Gauche \\ Case 247\\ 4 Place Jussieu,\
 75252 PARIS, Cedex 85, France}                                               

\email{marie-jose.bertin@imj-prg.fr} 

\email{odile.lecacheux@imj-prg.fr}

\date{\today}
\maketitle

\begin{abstract}

We compute the transcendental lattices of the singular $K3$ surfaces belonging to three pencils of $K3$ surfaces, namely the Ap\' ery-Fermi pencil with transcendental lattice $U \oplus \langle 12 \rangle$, the Verrill's pencil with transcendental lattice $U \oplus \langle 6 \rangle$ and another pencil linked to Verrill's pencil with transcendental lattice $U \oplus \langle 24 \rangle $. Many corollaries are deduced. For example, some singular $K3$ surfaces belong to different pencils or are Kummer surfaces of $K3$ surfaces of another pencil.

  \end{abstract}
\section{Introduction}
The first motivation to consider the three $K3$ families $(Y_k)$, $(Z_k)$ and $(X_k)$ was the computation of the Mahler measures of their defining polynomials respectively $P_k$, $Q_k$ and $Q'_k$. In a series of papers Bertin tried to generalize some formulae valid for a certain type of two variables polynomials defining elliptic curves, that is their Mahler measures are expressed in terms of a $L$-series of a modular form which is also the $L$-series of the elliptic curve itself.

For some values of $k$, the $K3$ surfaces $Y_k$, $Z_k$ and $X_k$ are singular (i.e. their Picard number is $20$) and the $L$-series of their transcendental lattices are given by the $L$-series of a modular form suspected to be the same as in the expression of the Mahler measure.

The family $(P_k)$  defining $(Y_k)$ with generic Picard number $19$ is the well known Ap\'ery-Fermi pencil defined by

\[P_k= x+\frac{1}{x}+ y+\frac{1}{y}+z+\frac{1}{z}-k. \]
It was studied first by Peters and Stienstra \cite{PS}, then used by Bertin \cite{Be1} \cite{Be2} and others \cite{Sa} to find explicit Mahler measures of some members and finally studied by Bertin and Lecacheux \cite{BL}.
In particuliar Peters and Stienstra \cite{PS} gave their generic transcendental lattice $U\oplus \langle 12 \rangle$.

The family $(Q_k)$ defining $(Z_k)$ is defined by 
\[Q_k=(x+y+z+1)(xy+xz+yz+xyz)-(k+4)xyz.\]
It has been studied concerning the Mahler measure of some members of the family, namely those defining singular $K3$ surfaces, by Bertin \cite{Be1} \cite{Be2}, Samart \cite{Sa} but the link between the Mahler measure and the $L$-series of the $K3$ surface was only considered by Bertin \cite{Be2} and her co-authors \cite{BFFLM}.

The family $(Q'_k)$  defining $(X_k)
$ with generic Picard number $19$ is Verrill's family \cite{Ve2} \cite{Ve3} with generic transcendental lattice $U\oplus \langle 6 \rangle$ and defined by
\[Q'_k=(X+XY+XYZ+1)(1+Z+ZY+ZXY)-(k+4)XYZ.\] 


The Mahler measure of $Q_k$ is in fact the same as the Mahler measure of $Q'_k$ since the change variables $x=X$, $y=XY$, $z=XYZ$ transforms $Q_k$ into $Q'_k$ and thus gets the Mahler measure inchanged.
Hence we can deduce that the generic Picard number of $Z_k$ is $19$ and that $Z_k$ and $X_k$ are singular $K3$ surfaces (i.e. with Picard number $20$) for the same values of $k$. These values and those corresponding to singular $K3$ surfaces $Y_k$ have been computed long ago by Boyd \cite{Bo}.

It leads also to our first question: is there a link between the two families $(Z_k)$ and $(X_k)$? That is the object of section $3$ where we prove the following theorem.
\begin{theorem}
  1) The transcendental lattice of the generic member $Z_k$ is $U\oplus \langle 24 \rangle $.

  2) There is a genus $1$ fibration of $Z_k$ whose Jacobian surface $J_k$ is a $K3$ surface of the Verrill's family with generic transcendental lattice $U\oplus \langle 6 \rangle$.

\end{theorem}

We are also interested in another question: among the singular members of these families, which ones belong to at least two families?

This question was suggested by an observation concerning the Mahler measures of their corresponding polynomials. In return, the answer to the question may give an expression of the Mahler measure in terms of the $L$-series of the $K3$ surface.

So all along sections $4$ to $6$ we shall prove the following theorem which completes some partial results already known. For simplification we adopt Shimada and Zhang's notation \cite{SZ}, that is $[a \quad b \quad c]$ means the transcendental lattice $\begin{pmatrix}
  a & b \\
  b & c
\end{pmatrix}$.

\begin{theorem}
 The transcendental lattices of the singular members of the previous families are given in the following table. In case we could not compute the transcendental lattice, we give the conjectural
discriminant of the $K3$ surface obtained from results on Mahler's measure of the defining polynomial
of the surface.
\begin{center} 
 \begin{tabular}{|l|l|l|l|l|l|}
 \hline
 $Y_0$ & $[4 \quad 2 \quad 4]$ & $Z_{-36}$ &$?12a^2$  & $X_{-36}$ & $[2 \quad 0 \quad 6 ]$\\
 
 $Y_2$ & $[2 \quad 0 \quad 4]$ & $Z_{-12}$ & $[10 \quad 2 \quad 10]$ & $X_{-12}$ & $[4 \quad 0 \quad 6 ]$\\
 
  $Y_3$ & $[2 \quad 1 \quad 8]$ & $Z_{-6}$ & $[8 \quad 4 \quad 8]$ & $X_{-6}$ & $[6 \quad 0 \quad 8 ]$\\
 
 $Y_6$ & $[2 \quad 0 \quad 12]$ & $Z_{-3}$ & $[4 \quad 1 \quad 4]$ & $X_{-3}$ & $[6 \quad 0 \quad 10 ]$\\  

 $Y_{10}$ & $[6 \quad 0 \quad 12]$ & $Z_{0}$ &$[2 \quad 0 \quad 6]$  & $X_{0}$ & $[2 \quad 0 \quad 6 ]$\\ 
 
   $Y_{18}$ & $[10 \quad 0 \quad 12]$ & $Z_{4}$ & $[2 \quad 0 \quad 16]$ & $X_{4}$ & $[2 \quad 0 \quad 4 ]$\\
  
   $Y_{102}$ & $[12 \quad 0 \quad 26]$ & $Z_{12}$ & $[2 \quad 0 \quad 24]$  & $X_{12}$ & $[2 \quad 1 \quad 2 ]$\\
   
    $Y_{198}$ & $[12 \quad 0 \quad 34]$ & $Z_{60}$ &$?15a^2$  & $X_{60}$ & $[4 \quad 1 \quad 4 ]$\\
    
     $Y_{k^2=20}$ & $[2 \quad 0 \quad 10]$ & &  &  & \\
     
 $Y_{k^2=54}$ & $[4 \quad 0 \quad 12]$ & &  &  & \\
 
 $Y_{k^2=-12}$ & $[6 \quad 0 \quad 6]$ & &  &  & \\

   $Y_{k^2=-45}$ & $[8 \quad 2 \quad 8]$ & &  &  & \\
 \hline

\end{tabular}
\end{center}

\end{theorem} 
 
The formulation $?12a^2$ (resp. $ ?15a^2 $), $a \in \mathbb Z$, means we suspect only the discriminant of the $K3$
surface with such a value.

The method of computation of the transcendental lattices of singular $K3$ surfaces $Y_k$  and $X_k$ is
similar. It uses the fact that the family is modular, meaning both that $k$ is a modular function and
their Picard-Fuchs equations have modular solutions. Since we ignore such a property for family
$(Z_k)$, we have to deduce the transcendental lattices from the N\'eron-Severi lattices thus require
equations to compute the transcendental lattices of some members.
Finally, in the last section we deduce which $K3$ surfaces belong to several families.
Computations were performed using partly the computer algebra system PARI \cite{PA}, partly Sage
\cite{Sage} and mostly the computer algebra system MAPLE and the Maple Library ``Elliptic Surface
Calculator'' written by Kuwata \cite{Ku}.
The first author is particularly indebted to David Boyd who computed which $k$ correspond to singular $K3$
surfaces and allowed her to perform the first computations of Mahler measures.

Although motivated by quite different aims, Sarti's paper \cite{Sar} must be cited here. In particuliar we notice in her work the two singular $K3$ surfaces with transcendental lattices $[4 \quad 1 \quad 4]$ and $[6 \quad 0 \quad 10]$ which are respectively the surfaces $Z_{-3}$ (or $X_{60}$) and $X_{-3}$ in our work.

We thank Alessandra Sarti for catching our attention on these results.

\section{Background}
For definitions and main properties concerning elliptic K3 surfaces see Sch\"{u}tt and Shioda's paper
\cite{SS}.

\subsection{Computation of  the N\'eron-Severi lattice}
Recall the following Shioda's result \cite{Shio}: 
Let $(S,\Phi, \mathbb P^1)$ be an elliptic surface with a section $\Phi$, without exceptional curves of first kind.

Denote by $NS(S)$ the group of algebraic equivalence classes of divisors of $S$.

Let $u$ be the generic point of $\mathbb P^1$ and $\Phi^{-1}(u)=E$ the elliptic curve defined over $K=\mathbb C(u)$ with a $K$-rational point $o=o(u)$. Then, 
$E(K)$ is an abelian group of finite type provided that $j(E)$ is transcendental over $\mathbb C$.
Let $r$ be the rank of $E(K)$ and $s_1,..., s_r$ be generators of $E(K)$ modulo torsion. Besides, the torsion group  $E(K)_{tors}$ is generated by at most two elements $t_1$ of order $e_1$ and $t_2$ of order $e_2$ such that $1\leq e_2$, $e_2 |e_1$ and $\mid E(K) _{tors}\mid =e_1e_2$.

The group $E(K)$ of $K$-rational points of  $E$ is canonically identified with the group of sections of $S$ over $\mathbb P^1(\mathbb C)$.

For $s\in E(K)$, we denote by $(s)$ the curve image in $S$ of the section corresponding to $ s$.

Let us define
$$D_{\alpha}:=(s_{\alpha})-(o) \,\,\,\,\,\,1\leq\alpha \leq r$$
$$D'_{\beta}:=(t_{\beta})-(o) \,\,\,\,\,\,\beta =1,2.$$                           

Consider now the singular fibers of $S$ over
$\mathbb P^1$. We set
$$\Sigma:=\{v\in \mathbb P^1  / C_v=\Phi^{-1}(v)\,\,\,\, {\hbox {be a singular \
fiber}} \}$$
and for each $v\in \Sigma$, $\Theta_{v,i}$, $0\leq i \leq m_v-1$, the
$m_v$ irreducible components of $C_v$.

Let $\Theta_{v,0}$ be the unique component of $C_v$ passing through $o(v)$.

One gets
$$C_v=\Theta_{v,0} +\sum_{i\geq 1}\mu_{v,i}\Theta_{v,i},\,\,\,\,\,\,\,\mu_{v,i}\
\geq 1.$$
Let $A_v$ be the matrix of order $m_v-1$ whose entry of index
$(i,j)$ is $(\Theta_{v,i}\Theta_{v,j})$, $i,j\geq 1$, where $(DD')$
is the intersection number of the divisors $D$ et $D'$ along
$S$. Finally $f$ will denote a non singular fiber,
i.e. $f=C_{u_0}$ for $u_0\notin \Sigma$.

\begin{theorem}

The N\'eron-Severi group $NS(S)$ of the elliptic surface $S$
is generated by the following divisors
$$f, \Theta_{v,i} \,\,\,\,\,\,(1\leq i \leq                                     
m_v-1,\,\,\,\,v\in\Sigma)$$
$$(o), D_{\alpha}\,\,\,\,\,\,1\leq \alpha \leq r,                               
\,\,\,\,D'_{\beta}\,\,\,\,\beta =1,2.$$

The only relations between these divisors are at most two relations
$$e_{\beta}D'_{\beta}\approx e_{\beta}(D'_{\beta} (o))f+\sum_{v\in
\Sigma
}(\Theta_{v,1},...,\Theta_{v,m_v-1})e_{\beta}A_v^{-1}\left ( \begin{array}{l}
(D'_{\beta}\Theta_{v,1})\\.\\
.\\
.\\
                                                               (D'_{\beta}\Theta_{v,m_v-1}
                                                               
)
\end{array}
\right )
$$
where $\approx$ stands for the algebraic equivalence.

\end{theorem}

\subsection{Computation of the transcendental lattice}
We recall some results from Shimada \cite{Shim}.

Let $L$ be an even lattice. The abelian group $G_L:=L^*/L$ is called the discriminant group of $L$. Moreover the symmetric bilinear form on $L$ extends to a $\mathbb Q$-valued symmetric bilinear form on $L^*$. Hence we get on $G_L$ the following quadratic norm $q_L$

\[                                                                                                                                                                                                
\begin{tabular}{lllll}                                                                                                                                                                            
        $q_L:$& $G_L$ & $\rightarrow$& $\mathbb Q/2 \mathbb Z$& \\                                                                                                                                
        & $\bar{x}$ & $\mapsto$ & $x^2$ & (mod. $2\mathbb Z)$.\\                                                                                                                                   
\end{tabular}                                                                                                                                                                                     
\]
The pair $(G_L,q_L)$ is by definition the discriminant form of the lattice $L$.

When the lattice $L$ is given by its Gram matrix, we can compute its discriminant form as explained in Shimada's lemma \cite{Shim}.

\begin{lemma}\label{lem:2.1}
Let $A$ the Gram matrix of $L$ and $U$, $V \in Gl_n(\mathbb Z)$ such that

\[ UAV=D=\begin{pmatrix}                                                                                                                                                                          
        d_1 &  & 0\\                                                                                                                                                                              
         & \ddots & \\                                                                                                                                                                            
        0 &  & d_n\\                                                                                                                                                                              
\end{pmatrix}                                                                                                                                                                                     
\]
with $1=d_1=\ldots =d_k < d_{k+1} \leq \ldots \leq d_n$. Then

\[G_L\simeq \oplus_{i>k} \mathbb Z/(d_i).\]
Moreover the $i$th row vector of $V^{-1}$, regarded as an element of $L^*$ with respect to the dual basis $e_1^*$, ..., $e_n^*$ generate the cyclic group $\mathbb Z/(d_i)$.
\end{lemma}

Define $H^2(X):=H^2(X,\mathbb Z)/Tors$, 

\[NS(X):=H^2(X)\cap H^{1,1}(X)\]
\[T(X):=NS(X)^{\perp}.\]
The transcendental lattice $T(X)$ is a birational invariant of algebraic surfaces satisfying the following relation between the discriminant forms of $T(X)$ and $NS(X)$

\[(G_{T(X)},q_{T(X)})\equiv (G_{NS(X)},-q_{NS(X)}).\]

\section{Proof of theorem 1.1}
\begin{notation}
The singular fibers of type $\,I_{n},I_{n}^*,IV^{\ast}, ...$ at $t=t_{1},.,t_{m}$
or at  roots of a polynomial $p(t)$ of degree $m$ are denoted $mI_{n}%
(t_{1},..,t_{m})$ or $mI_{n}(p(t)).$
\end{notation}

\begin{theorem}
  (1) The transcendental lattice of the generic member $Z_k$ is $U \oplus \langle 24 \rangle$.
  
(2) There is a genus $1$ fibration of $Z_k$ whose Jacobian surface $J_k$ is a $K3$ surface of the Verrill's family
with generic transcendental lattice $U \oplus \langle 6\rangle$.
\end{theorem}

\begin{proof}
 (1) We consider the family of $K3$ surfaces $(Z_k)$ defined by the polynomials
 \[Q_k = (x + y + z + 1)(xy + xz + yz + xyz)-(k + 4)xyz.\]
 \begin{lemma} \label{lem:3.1}
There is a generic rank $0$  elliptic fibration of the generic member $Z_k$ given by the Weierstrass equation
\[(F_k) \qquad y^2=x^3+((k^2-24)t^2+2(k-2)(k+4)^2t+k(k+4)^3)x^2-16t^4(t+k+3)x\]

with $2$-torsion and singular fibers $III^*$ , $I_8$ , $I_3$ , $I_2$ , $2I_1$ .
\end{lemma}

\begin{proof}
  Taking for elliptic parameter $t=x+y+z$, with the change variables
  \[x=\frac{(Y-X^2+Yt)}{Y(k+3-t)}, \qquad y=\frac{Y+(t+1)X}{X(
      k+3-t)},\]
  it follows the Weierstrass equation
  \[Y^2+(t^2-kt+3)XY-(t+1)^2(kt-(t-1)^2)Y=X^3.\]

  This defines a generic rank $1$ elliptic fibration with $6$-torsion and singular fibers $I_6(-1)$, $2I_3((t-1)^2-kt)$, $2I_2(k+3,0)$, $2I_1((t+3)^2-kt)$, $I_6(\infty)$.

  Now the new elliptic parameter $m=\frac{Y+(k+4)X}{(t+1)^2}$ leads to a rank $0$ elliptic fibration with a $2$-torsion section. First by elimination of $Y$, then with the new variable $W$ defined by $X=(t+1)^2W$, it follows an equation of type $a(W)t^2+b(W)t+c(W)=0$. Completing the square we deduce the equation
  \[Q^2=(W+1)(W(k+4)-m)(W^2(k(k+4)^2-4m)+W(k((k+4)^2-m(k+8))-16m)-m(k^2+4k-4m)).\]

  This is the equation of a quartic with the point $[W=-1,Q=0]$. Using a classical change variables $W=-1+\frac{1}{x}$, $Q=\frac{y}{x^2}$, then $x=\frac{k+4}{m+k+4}+X$ and finally

  \[X=-\frac{U}{4m(k+m+3)(k+m+4)} \qquad y=\frac{Y}{4m(k+m+3)(k+m+4)}\]

  we get
  \[Y^2=U^3+\left ((k^2-24)m^2+2(k-2)(k+4)^2m+k(k+4)^3\right )U^2-16Um^4(m+k+3),\]
which is precisely the equation $(F_k)$.

\end{proof}

Starting with the equation $(F_k)$ and using a Shioda and Inose result \cite{SI}, the discriminant of $Z_k$ is equal to

\[ \Delta = \frac{2 \times 8 \times 3 \times 2}{4}=24.\]

Our aim is to compute the transcendental lattice of the Picard number $19$ surface $Z_k$.

Denote $\theta_i$ $1\leq i \leq 7$, $\eta_i$, $1\leq i \leq 7$, $\alpha_i$, $i=1,2$, $\beta_1$ the generators corresponding respectively to the singular fibers $I_8$, $III^*$, $I_3$, $I_2$ and $s_2$ the $2$-torsion section.

The relation is the following:
\[s_2 \approx \frac{1}{2} \theta_1 +\theta_2+\frac{3}{2}\theta_3+2\theta_4 +\frac{3}{2} \theta_5 +\theta_6+\frac{1}{2} \theta_7+\frac{3}{2}\eta_1+2\eta_2+\frac{5}{2}\eta_3+3\eta_4+2\eta_5+\eta_6+\frac{3}{2}\eta_7+\frac{1}{2}\beta_1   \]

and since $s_2$ cuts $\theta_4$, cuts $\eta_1$, does not cut $\alpha_i, i=1,2$, but cuts $\beta_1$, taking the divisors in the following order $(0)$, $f$, $s_2$, $\theta_i, i=2,..,7$, $\eta_i, i=1,..,7$, $\alpha_i, i=1,2$, $\beta_1$, we get the following Gram matrix

\[ \left (\begin{smallmatrix}
	-2& 1 & 0 & 0 & 0 & 0 & 0 & 0 & 0 & 0 & 0 & 0 & 0 & 0 & 0 & 0 & 0 & 0 & 0\\
	1 & 0 & 1 & 0 & 0 & 0 & 0 & 0 & 0 & 0 & 0 & 0 & 0 & 0 & 0 & 0 & 0 & 0 & 0\\
	0 & 1 & -2 & 0 & 0 & 1 & 0 & 0 & 0 & 1 & 0 & 0 & 0 & 0 & 0 & 0 & 0 & 0 & 1\\
	0& 0 & 0 & -2 & 1 & 0 & 0 & 0 & 0 & 0 & 0 & 0 & 0 & 0 & 0 & 0 & 0 & 0 & 0\\
	0 & 0 & 0 & 1 & -2 & 1 & 0 & 0 & 0 & 0 & 0 & 0 & 0 & 0 & 0 & 0 & 0 & 0 & 0\\
	0& 0 & 1 & 0 & 1 & -2 & 1 & 0 & 0 & 0 & 0 & 0 & 0 & 0 & 0 & 0 & 0 & 0 & 0\\
	0 & 0 & 0 & 0 & 0 & 1 & -2 & 1 & 0 & 0 & 0 & 0 & 0 & 0 & 0 & 0 & 0 & 0 & 0\\
	0& 0 & 0 & 0 & 0 & 0 & 1 & -2 & 1 & 0 & 0 & 0 & 0 & 0 & 0 & 0 & 0 & 0 & 0\\
	0 & 0 & 0 & 0 & 0 & 0 & 0 & 1 & -2 & 0 & 0 & 0 & 0 & 0 & 0 & 1 & 0 & 0 & 0\\
	0 & 0 & 1 & 0 & 0 & 0 & 0 & 0 & 0 & -2 & 1 & 0 & 0 & 0 & 0 & 0 & 0 & 0 & 0\\
	0 & 0 & 0 & 0 & 0 & 0 & 0 & 0 & 0 & 1 & -2 & 1 & 0 & 0 & 0 & 0 & 0 & 0 & 0\\
	0 & 0 & 0 & 0 & 0 & 0 & 0 & 0 & 0 & 0 & 1 & -2 & 1 & 0 & 0 & 0& 0 & 0 & 0\\
	0 & 0 & 0 & 0 & 0 & 0 & 0 & 0 & 0 & 0 & 0 & 1 & -2 & 1 & 0 & 1 & 0 & 0 & 0\\
	0 & 0 & 0 & 0 & 0 & 0 & 0 & 0 & 0 & 0 & 0 & 0 & 1 & -2 & 1 & 0 & 0 & 0 & 0\\
	0 & 0 & 0 & 0 & 0 & 0 & 0 & 0 & 0 & 0 & 0 & 0 & 0 & 1 & -2 & 0 & 0 & 0 & 0\\
	0 & 0 & 0 & 0 & 0 & 0 & 0 & 0 & 0 & 0 & 0 & 0 & 1 & 0 & 0 & -2 & 0 & 0 & 0\\
	0 & 0 & 0 & 0 & 0 & 0& 0 & 0 & 0 & 0 & 0 & 0 & 0 & 0 & 0 & 0 & -2 & 1 & 0\\
	0 & 0 & 0 & 0 & 0 & 0 & 0 & 0 & 0 & 0 & 0 & 0 & 0 & 0 & 0 & 0 & 1 & -2 & 0\\
	0 & 0 & 1 & 0 & 0 & 0 & 0 & 0 & 0 & 0 & 0 & 0 & 0 & 0 & 0 & 0 & 0 & 0 & -2\\
      \end{smallmatrix}
      \right )
      \]

      Its determinant is $24$.

According to lemma \ref{lem:2.1} and taking for $A$ the Gram matrix of the N\'eron-Severi lattice, we find $D=I_{18} \oplus 24$ and the last row of $V^{-1}$ is 
\[B=(-3,0,0,0,0,-3,3,0,0,0,3,-6,3,0,0,0,-1,1,-3).\]
Thus we can compute the norm of $B$ 
\[BA^{-1} B^t=-\frac{745}{24}=-\frac{25}{24} \,\,\text{mod. } 2.\]

We deduce 
\[(G_{NS(X)},q_{NS(X)})=(\mathbb Z/24 \mathbb Z) (-\frac{25}{24})\]
\[(G_{T(X)},q_{T(X)})=(\mathbb Z/24 \mathbb Z) (\frac{25}{24}).\]

With the same method if $T'=U\oplus \langle 24 \rangle $we find 
\[(G_{T'},q_{T'})=(\mathbb Z/24 \mathbb Z) (\frac{1}{24})\]
which means that a generator $\alpha$ of the cyclic group $\mathbb Z/24 \mathbb Z$ has norm $\frac{1}{24}$. Replacing the generator $\alpha $ by $5\alpha$ we find
\[(G_{T'},q_{T'})=(\mathbb Z/24 \mathbb Z) (\frac{25}{24}).\]
Thus $T(X)=T'$.

(2)  We start with the Weierstrass equation of $Z_k$
 \[(F_k) \qquad y^2=x^3+((k^2-24)t^2+2(k-2)(k+4)^2t+k(k+4)^3)x^2-16t^4(t+k+3)x.\]

 With the new parameter $m=\frac{-1}{4\left(  k+4\right)  ^{2}}\frac{x}{\left(                  
t+k+3\right)  }$ and with $y=4W\left(  t+k+3\right)  \left(  k+4\right)  $ we
get%
\begin{align*}                                                                                 
W^{2} &  =\\                                                                                   
&  -4t^{4}m+(k^{2}-24)(k+4)^{2}m^{2}t^{2}+(4m+2(k-2))(k+4)^{4}m^{2}t+(                         
4(k+3)m+k(k+4))(k+4)^{4}m^{2}.%
\end{align*}

By classical trick its Jacobian $J_{k}$ has Weierstrass equation changing $m$
in $t$%

\[                                                                                             
y^{2}=x^{3}+t^{2}(k^{2}-24)x^{2}+16{t}^{3}(4 (k+3 )t-k (4+k ) )x+64t^{5}( (4+k                 
)^{2}{t}^{2}-(3{k}^{2}+24k+40)t+ (4+k ) (3k+4)).                                                
\]

Finally the change $x$ into $x+8t^{2}$ leads to the equation
\[                                                                                             
(J_{k})\qquad y^{2}=x^{3}+t^{2}k^{2}x^{2}+16t^{3}k(k+4)(t-1)x+64t^{5}%
(k+4)^{2}(t-1)^{2},                                                                            
\]
with singular fibers $III^{\ast}(0)$, $I_{3}(1)$, $2I_{1}(432\left(                             
k+4\right)  ^{2}t^{2}+\left(  k^{6}-72(4+k)(k^{3}+6k+24)\right)                                
t+64k^{3}(4+k))$, $II^{\ast}(\infty)$ hence of type $A_{2}E_{7}E_{8}$.

Now all the elliptic fibrations of the $K3$ surface with discriminant $6$ and transcendental lattice
\begin{center}
 $\begin{pmatrix}                                                                                                                                                                                 
        0 & 0 & 1\\                                                                                                                                                                               
        0 & 6 & 0\\                                                                                                                                                                               
        1 & 0 & 0\\                                                                                                                                                 
\end{pmatrix}$
\end{center}
are given by the primitive embeddings of $E_6 \oplus A_1$ into the Niemeier lattices (Sterk \cite{St}). Indeed taking the primitive embedding of $E_6 \oplus A_1$ into the Niemeier lattice $E_8^3$ defined by  the primitive embeddings $E_6 \hookrightarrow E_8$ and $A_1 \hookrightarrow E_8$, we get the above elliptic fibration since $(E_6)_{E_8}^{\perp}=A_2$ and $(A_1)_{E_8}^{\perp}=E_7$ \cite{Ni}. It corresponds to the second fibration among the ten jacobian elliptic fibrations obtained by Sterk \cite{St}.

\end{proof}

\section{Transcendental lattices of singular $K3$-surfaces of Verrill's family $(X_k)$}
The singular $K3$ surfaces of Verrill's family $(X_k)$ are given by the imaginary quadratic $\tau$ \cite{Ve1} such that
\[t=\frac{\eta(3\tau)^4 \eta(12\tau)^8\eta(2\tau)^{12}}{\eta(\tau)^4 \eta(4 \tau)^8\eta(6\tau)^{12}}, \qquad k=-(t+\frac{1}{t})-2.\]
They have been computed by Boyd \cite{Bo} and are given in the following table.
\begin{center}
  \begin{tabular}{|l|l|l|}
    \hline
    k & $\tau$ & Equation of $\tau$  \\ \hline
  $-36$ & $\frac{3+\sqrt{-3}}{3}$ & $3 \tau^2-6\tau+4=0$ \\ \hline
    $-12$ & $\frac{\sqrt{-6}}{6}$ & $6\tau ^2+1=0$ \\ \hline
    $-6$ &  $\frac{\sqrt{-3}}{6}$ & $12\tau^2+1=0$ \\ \hline
    $-3$ & $\frac{3+\sqrt{-15}}{24}$ & $24\tau^2-6\tau+1=0$ \\ \hline
    $0$ & $\frac{3+\sqrt{-3}}{12}$ & $12\tau^2-6\tau+1=0$ \\ \hline
    $4$ & $\frac{-2+\sqrt{-2}}{6}$ & $6\tau^2+4\tau+1=0$ \\ \hline
    $12$ & $\frac{3+\sqrt{-3}}{6}$ & $3\tau^2-3\tau+1=0$ \\ \hline
    $60$ & $\frac{3+\sqrt{-15}}{6}$ & $3\tau^2-3\tau+2=0$ \\ \hline

\end{tabular}
\end{center}

Let us recall how to compute the transcendental lattices of the singular $K3$ surfaces of Verrill's family \cite{Ve3}.

If $Y_{\lambda}$ denotes a $K3$ surface, up to a scalar, there exists a unique $2$-form $\omega_{\lambda} \in H^{2,0}(Y_{\lambda},\mathbb C)$.
If $\{ \gamma_{1,\lambda}, \cdots , \gamma_{22,\lambda} \}$ is a basis of $H_2(Y_{\lambda},\mathbb Z)$ then  $\{ \gamma_{1,\lambda}^*, \cdots , \gamma_{22,\lambda}^* \}$ is the corresponding dual basis in 
 $H^2(Y_{\lambda},\mathbb C)$, dual for the intersection product. That means, if $\gamma \in H_2(Y_{\lambda},\mathbb Z)$, then  $\gamma^* \in H^2(Y_{\lambda},\mathbb Z)$ is defined for every $\alpha\in H_2(Y_{\lambda}, \mathbb C))$ by
\[\int_{\alpha} \gamma^* =(\gamma \cdot \alpha)\]
where $(\gamma \cdot \alpha)$ is the intersection number of the two cycles $\alpha$ and $\gamma$. The intersection product of homological cycles satisfies the Poincar\'e duality for the exterior product  i.e. $\forall \alpha, \gamma \in H_2(Y_{\lambda},\mathbb Z)$,
it follows
\[(\gamma \cdot \alpha )=\int_{Y_{\lambda}}\alpha^* \wedge \gamma^*.\]

Hence the dual basis intersection matrix is also the intersection matrix of the  transcendental cycles lattice $T(Y_{\lambda})=\text{Pic}(Y_{\lambda})^{\perp}$. For Verrill's family, the Gram matrix of its transcendental lattice in the basis  $\{ \gamma_{1}, \gamma_2,\gamma_3 \}$ is given by
\[                                                                                                                                                                                                                 
T= \left(                                                                                                                                                                                                          
\begin{matrix}                                                                                                                                                                                                     
0 & 0 & 1 \\                                                                                                                                                                                                       
0 & 6 & 0 \\                                                                                                                                                                                                      
1 & 0 & 0\\

\end{matrix}                                                                                                                                                                                                       
\right).                                                                                                                                                                                                           
\]

Since the integral of $\omega_{\lambda}$ along any algebraic cycle is zero, it follows from the identification of $H_2(Y_{\lambda},\mathbb Z) \otimes \mathbb C$ with $H^2(Y_{\lambda},\mathbb C)$ that $\omega_{\lambda} \in T(Y_
{\lambda})\otimes \mathbb C$.

Hence $\omega_{\lambda}\in H^{2,0}(Y_{\lambda},\mathbb C)$ can be written as
\[\omega_{\lambda(\tau)}=a(\tau)\gamma_1^*+b(\tau)\gamma_2^*+c(\tau)\gamma_3^*\]
with
\[6 b(\tau)^2+2a(\tau)c(\tau)=0,\]
since $\omega_{\tau}\wedge \omega_{\tau} \in H^{(4,0)}=(0)$.

It follows that
\[\omega(\tau)=G(\tau)\gamma_1^*+\tau G(\tau)\gamma_2^* -3\tau^2 G(\tau)\gamma_3^*\]
belongs to $T(X_{\lambda(\tau)})$ where $\{ G(\tau),\tau G(\tau), \tau^2 G(\tau) \}$ satisfies the Picard-Fuchs differential equation of periods of Verrill's family.

Thus the fact that the $K3$ surface $X_{\lambda}$ has Picard number $20$ is equivalent to the fact that there exists a vector $p\gamma_1+q\gamma_2+r\gamma_3 \in T_{\lambda}$ becoming algebraic, hence satisfying $\int_{p\gamma_1+q\gamma_2+r\gamma_3}\omega(\tau)=0$, i.e.
\[-3p\tau^2+6q\tau+r=0.\]

This leads to the determination of the Gram matrix of the transcendental lattices of the singular $K3$ surfaces of Verrill's family.

\subsection{Surface $X_{-36}$}
The corresponding $\tau$ satisfies $3\tau^2-6\tau+4=0$. Hence $p=-1$, $q=-1$, $r=4$ and the vector $-\gamma_1-\gamma_2+4\gamma_3$ becomes algebraic. It follows
that the vectors $\langle \gamma_1+\gamma_2-2\gamma_3,\gamma_2-6\gamma_3 \rangle $ generate the transcendental lattice and
\[T(X_{-36})= [2 \quad 0 \quad 6 ].\]

\subsection{Surface $X_{-12}$}
The corresponding $\tau$ satisfies $6\tau^2+1=0$. Hence $p=-2$, $q=0$, $r=1$ and the vector $-2\gamma_1+\gamma_3$ becomes algebraic. It follows that the vectors
$\langle 2\gamma_1+\gamma_3, \gamma_2 \rangle$ generate the transcendental lattice and
\[T(X_{-12})=[4 \quad 0 \quad 6 ].\]

\subsection{Surface $X_{-6}$}
The corresponding $\tau$ satisfies $12\tau^2+1=0$. Hence $p=-4$, $q=0$, $r=1$ and the vector $-4\gamma_1+\gamma_3$ becomes algebraic. It follows that the vectors $\langle 4\gamma_1+\gamma_3, \gamma_2 \rangle$ generate the transcendental lattice and
\[T(X_{-6})=[ 6 \quad 0 \quad 8].\]

\subsection{Surface $X_{-3}$}
The corresponding $\tau$ satisfies $24\tau^2+6\tau+1=0$.

Hence $p=-8$, $q=-1$, $r=1$ and the vector $-8\gamma_1-\gamma_2+\gamma_3$ becomes algebraic. It follows that the vectors
$\langle \gamma_2+6\gamma_1, -\gamma_2+2\gamma_1+\gamma_3 \rangle$ generate the transcendental lattice and
\[T(X_{-3})=[6 \quad 0 \quad 10].\]

\subsection{Surface $X_{0}$}
The corresponding $\tau$ satisfies $12\tau^2-6\tau+1=0$.

Hence $p=-4$, $q=-1$, $r=1$ and the vector $-4\gamma_1-\gamma_2+\gamma_3$ becomes algebraic. It follows that the vectors $\langle -2\gamma_1- \gamma_2+\gamma_3, \gamma_2+6\gamma_1 \rangle$ generate the transcendental lattice and
\[T(X_{0})=[2 \quad 0 \quad 6].\]

\subsection{Surface $X_{4}$}
The corresponding $\tau$ satisfies $6\tau^2+4\tau+1=0$.

Hence $p=-6$, $q=2$, $r=3$ and the vector $-6\gamma_1+2\gamma_2+3\gamma_3$ becomes algebraic. It follows that the vectors
$\langle \gamma_1+2\gamma_3, \gamma_2-2\gamma_1+\gamma_3 \rangle$ generate the transcendental lattice and
\[T(X_{4})=[2 \quad 0 \quad 4].\]

\subsection{Surface $X_{12}$}
The corresponding $\tau$ satisfies $3\tau^2-3\tau+1=0$.

Hence $p=-2$, $q=-1$, $r=2$ and the vector $-2\gamma_1-\gamma_2+2\gamma_3$ becomes algebraic. It follows that the vectors
$\langle \gamma_1+\gamma_3, \gamma_2+2\gamma_1-\gamma_3 \rangle$ generate the transcendental lattice and
\[T(X_{12})=[2 \quad 1 \quad 2 ].\]

\subsection{Surface $X_{60}$}
The corresponding $\tau$ satisfies $3\tau^2-3\tau+2=0$.

Hence $p=-2$, $q=-1$, $r=4$ and the vector $-2\gamma_1-\gamma_2+4\gamma_3$ becomes algebraic. It follows that the vectors $\langle \gamma_1+2\gamma_3, \gamma_2+\gamma_1-\gamma_3 \rangle $  generate the transcendental lattice and
\[T(X_{60})=[4 \quad 1 \quad 4].\]

\section{Transcendental lattices of singular $K3$ surfaces of the Ap\'ery-Fermi's family}

The singular $K3$ surfaces of the Ap\'ery-Fermi's family $(Y_k)$ are given by the imaginary quadratic $\tau$ \cite{PS} such that
\[t=\left (\frac{\eta(\tau)\eta(6\tau)}{\eta(2\tau \eta(3\tau)} \right )^6, \qquad k=t+\frac{1}{t}.\]

They have been computed by Boyd \cite{Bo} and are given in the following table
\begin{center}
  \begin{tabular}{|l|l|l|}
    \hline
    k & $\tau$ & Equation of $\tau$  \\ \hline
     $0$ & $\frac{-3+\sqrt{-3}}{6}$ & $3\tau ^2+3\tau+1=0$ \\ \hline
  $2$ & $\frac{-2+\sqrt{-2}}{6}$ & $6 \tau^2+4\tau+1=0$ \\ \hline
     $3$ &  $\frac{-3+\sqrt{-15}}{12}$ & $6\tau^2+3\tau+1=0$ \\ \hline
    $6$ & $\frac{\sqrt{-6}}{6}$ & $6\tau^2+1=0$ \\ \hline
    $10$ & $\frac{\sqrt{-2}}{2}$ & $2\tau^2+1=0$ \\ \hline
    $18$ & $\frac{\sqrt{-30}}{6}$ & $6\tau^2+5=0$ \\ \hline
    
    $102$ & $\frac{\sqrt{-6\times 13}}{6}$  &$6\tau ^2+13=0$ \\ \hline
    $198$ & $\frac{\sqrt{-17 \times 6}}{6}$ & $6\tau^2+17=0$ \\ \hline
    $2\sqrt{5}$ & $\frac{-1+\sqrt{-5}}{6}$ & $6\tau^2 +2\tau +1=0$ \\ \hline
    $3\sqrt{6}$ & $\frac{\sqrt{-3}}{3}$ & $3\tau^2 +1=0$ \\ \hline
    $2\sqrt{-3}$ & $\frac{-1+\sqrt{-1}}{2}$ & $2\tau^2+2\tau +1=0$ \\ \hline
    $3\sqrt{-5}$ & $\frac{-3+\sqrt{-15}}{6}$ & $3\tau^2+3\tau+2=0$ \\ \hline

\end{tabular}
\end{center}

For Ap\'ery-Fermi's family the Gram matrix of its transcendental lattice in the basis $\{ \gamma_1, \gamma_2, \gamma_3 \}$ is given by
\[                                                                                                                                                                                                                 
T= \left(                                                                                                                                                                                                          
\begin{matrix}                                                                                                                                                                                                     
0 & 0 & 1 \\                                                                                                                                                                                                       
0 & 12 & 0 \\                                                                                                                                                                                                      
1 & 0 & 0\\

\end{matrix}                                                                                                                                                                                                       
\right).                                                                                                                                                                                                           
\]

As previously, the fact that $Y_{\lambda}$ has Picard number $20$ is equivalent to the existence of a vector $p\gamma_1+q\gamma_2 +r\gamma_3 \in T_{\lambda}$ becoming algebraic, hence satisfying $\int_{p\gamma_1+q\gamma_2+r\gamma_3}\omega(\tau)=0$, that is
\[-6p\tau^2+12q\tau+r=0.\]
In various papers \cite{Be2} \cite{BFFLM}, it was proved

\[T(Y_0)=[4 \quad 2 \quad 4] \qquad T(Y_2)=[2 \quad 0 \quad 4]\qquad T(Y_6)=[2 \quad 0 \quad 12]\]

\begin{align*}                                                                                                                                                                                    
T(Y_3)=[2\quad 1 \quad 8] \qquad \quad T(Y_{10})=[6 \quad 0& \quad 12] \qquad \quad T(Y_{18})=[10 \quad 0 \quad 12]\\                                                                             
T(Y_{102})=[12 \quad 0 \quad 26]\qquad & \qquad  T(Y_{198})=[12 \quad 0 \quad 34].                                                                                                                
\end{align*}

We achieve with the remaining surfaces of the list.

\subsection{Surface $Y_{2\sqrt{5}}$} The corresponding $\tau$ satisfies $6\tau^2+2 \tau +1=0$. Hence $p=-6$, $q=1$, $r=6$ and the vector $-6\gamma_1 +\gamma_2+6\gamma_3$ becomes algebraic. It follows that the vectors
$\langle \gamma_1+\gamma_3, \gamma_2-\gamma_1+\gamma_3 \rangle$ generate the transcendental lattice and
\[T(Y_{2\sqrt{5}})= [2 \quad 0 \quad 10].\]

\subsection{Surface $Y_{3\sqrt{6}}$} The corresponding $\tau$ satisfies $3\tau^2 +1=0$.
Thus the vector $-\gamma_1 +2\gamma_3$ becomes algebraic. It follows that the vectors
$\langle \gamma_1+2\gamma_3, \gamma_2 \rangle $ generate the transcendental lattice and
\[T(Y_{3\sqrt{6}})=[4 \quad 0 \quad 12].\]

\subsection{Surface $Y_{2\sqrt{-3}}$} The corresponding $\tau$ satisfies $2\tau^2+2\tau +1=0$.
Thus the vector $-2\gamma_1+\gamma_2 +6\gamma_3$ becomes algebraic. It follows that the vectors 
$\langle \gamma_1+3\gamma_3, \gamma_2-\gamma_1+3\gamma_3 \rangle $ generate the transcendental lattice and
\[T(Y_{2\sqrt{-3}})=[6 \quad 0 \quad 6].\]

\subsection{Surface $Y_{3\sqrt{-5}}$} The corresponding $\tau$ satisfies $3\tau^2+3 \tau +2=0$.
Thus the vector $-2\gamma_1 +\gamma_2+8\gamma_3$ becomes algebraic. It follows that the vectors $\langle \gamma_1+4\gamma_3, \gamma_2-\gamma_1+2\gamma_3 \rangle$ generate the transcendental lattice and
\[T(Y_{3\sqrt{-5}})=[8 \quad 2 \quad 8].\]

\section{Transcendental lattices of singular $K3$ surfaces of the $(Q_k)$ family}

Let us recall a Weierstrass equation for the $K3$ surface $Z_k$, see lemma \ref{lem:3.1}.
\[(F_k) \qquad y^2=x^3+((k^2-24)t^2+2(k-2)(k+4)^2t+k(k+4)^3)x^2-16t^4(t+k+3)x.\]

\subsection{$T(Z_{-3})$}  A Wierstrass equation is given by
  \[(F_{-3}) \qquad y^2=x^3-x^2(15t^2+10t+3)-16t^5x\]
  with singular fibers $I_{10}(0)$,$I_3(-1)$, $2I_1(64t^2+33t+9)$, $III^*(\infty)$. The rank is $0$ and by Shimada and Zhang \cite{SZ} we get
  \[T(Z_{-3})=[4 \quad 1 \quad 4].\]

    \subsection{$T(Z_0)$}  A Weierstrass equation is given by
    \[(F_0) \qquad y^2=x^3-x^2(24t^2+64t)-16t^4(t+3)x\]
    with singular fibers $I_4^*(0)$, $I_3(-4)$, $I_2(-3)$, $III^*(\infty)$. The rank is $0$ and by Shimada and Zhang \cite{SZ},
    \[T(Z_0)=[2 \quad 0 \quad 6].\]

\subsection{$T(Z_{12})$}  A Wierstrass equation is given by
  \[(F_{12}) \qquad y^2=x^3+x^2(120t^2+5120t+49152)-16t^4(t+15)x\]
  with singular fibers $I_{8}(0)$, $I_3(-16)$, $2I_2(-15,-96)$, $III^*(\infty)$. The rank is $0$ and by Shimada and Zhang \cite{SZ} it follows
  \[T(Z_{12})=[2 \quad 0 \quad 24].\]

  \bigskip

  In the following specializations of $Z_k$ we order the singular fibers as in the generic case section $3$.

\subsection{$T(Z_{-6})$}

A Weierstrass equation is given by
\[(F_{-6}) \qquad y^2=x^3+x^2(3t^2-16t+12)-xt^4(t-3)\]
with singular fibers $I_8(0)$, $I_3(2)$, $I_2(3)$, $III^*(\infty)$.
The rank is $1$ and the infinite section given by the point $P$ generates the Mordell-Weil group. We obtain 
\[P=[16(t-3),4i(t-3)(t^2-24)] \]
and find this infinite section $P$ does not cut the zero section and cuts only $\eta_1$ and $\beta_1$.
We deduce the Gram matrix of the N\'eron-Severi lattice
\[NS=
  \left (
  \begin{smallmatrix}
    -2& 1 & 0 & 0 & 0 & 0 & 0 & 0 & 0 & 0 & 0 & 0 & 0 & 0 & 0 & 0 & 0 & 0 & 0& 0\\
     1 & 0 & 1 & 0 & 0 & 0 & 0 & 0 & 0 & 0 & 0 & 0 & 0 & 0 & 0 & 0 & 0 & 0 & 0&1\\
     0 & 1 & -2 & 0 & 0 & 1 & 0 & 0 & 0 & 1 & 0 & 0 & 0 & 0 & 0 & 0 & 0 & 0 & 1&0\\
    0& 0 & 0 & -2 & 1 & 0 & 0 & 0 & 0 & 0 & 0 & 0 & 0 & 0 & 0 & 0 & 0 & 0 & 0&0\\
    0 & 0 & 0 & 1 & -2 & 1 & 0 & 0 & 0 & 0 & 0 & 0 & 0 & 0 & 0 & 0 & 0 & 0 & 0&0\\
     0 & 0 & 1 & 0 & 1 & -2 & 1 & 0 & 0 & 0 & 0 & 0 & 0 & 0 & 0 & 0 & 0 & 0 & 0&0\\
	 0 & 0 & 0 & 0 & 0 & 1 & -2 & 1 & 0 & 0 & 0 & 0 & 0 & 0 & 0 & 0 & 0 & 0 & 0&0\\
	0& 0 & 0 & 0 & 0 & 0 & 1 & -2 & 1 & 0 & 0 & 0 & 0 & 0 & 0 & 0 & 0 & 0 & 0 &0\\
	0 & 0 & 0 & 0 & 0 & 0 & 0 & 1 & -2 & 0 & 0 & 0 & 0 & 0 & 0 & 1 & 0 & 0 & 0 & 0\\
	0 & 0 & 1 & 0 & 0 & 0 & 0 & 0 & 0 & -2 & 1 & 0 & 0 & 0 & 0 & 0 & 0 & 0 & 0 & 1\\
	0 & 0 & 0 & 0 & 0 & 0 & 0 & 0 & 0 & 1 & -2 & 1 & 0 & 0 & 0 & 0 & 0 & 0 & 0 & 0\\
	0 & 0 & 0 & 0 & 0 & 0 & 0 & 0 & 0 & 0 & 1 & -2 & 1 & 0 & 0 & 0& 0 & 0 & 0 & 0\\
	0 & 0 & 0 & 0 & 0 & 0 & 0 & 0 & 0 & 0 & 0 & 1 & -2 & 1 & 0 & 1 & 0 & 0 & 0 & 0\\
	0 & 0 & 0 & 0 & 0 & 0 & 0 & 0 & 0 & 0 & 0 & 0 & 1 & -2 & 1 & 0 & 0 & 0 & 0 & 0\\
	0 & 0 & 0 & 0 & 0 & 0 & 0 & 0 & 0 & 0 & 0 & 0 & 0 & 1 & -2 & 0 & 0 & 0 & 0 & 0\\
	0 & 0 & 0 & 0 & 0 & 0 & 0 & 0 & 0 & 0 & 0 & 0 & 1 & 0 & 0 & -2 & 0 & 0 & 0 & 0\\
	0 & 0 & 0 & 0 & 0 & 0& 0 & 0 & 0 & 0 & 0 & 0 & 0 & 0 & 0 & 0 & -2 & 1 & 0 & 1\\
	0 & 0 & 0 & 0 & 0 & 0 & 0 & 0 & 0 & 0 & 0 & 0 & 0 & 0 & 0 & 0 & 1 & -2 & 0 & 0\\
	0 & 0 & 1 & 0 & 0 & 0 & 0 & 0 & 0 & 0 & 0 & 0 & 0 & 0 & 0 & 0 & 0 & 0 & -2 & 0\\
	0& 1 &  0 & 0 & 0 & 0 & 0 & 0 & 0 & 1 & 0 & 0 & 0 & 0 & 0 & 0 & 1 & 0 & 0 & -2\\
        \end{smallmatrix}
\right )
\]

and its determinant $-48$. The elementary divisors are $4$ and $12$. Applying lemma \ref{lem:2.1} we find
\[G_{NS}=\mathbb Z/(4) \oplus \mathbb Z/(12).\]

Moreover the vector
\[B_4=[-3,0,-1,0,0,0,0,0,1,0,0,-3,3,-3,3,0,-3,3,-2,3]\]
with quadratic norm $-\frac{1}{2}$ generates the cyclic group $\mathbb Z/(4)$ and the vector

\[B_{12}=[1,0,0,0,0,0,0,0,0,0,0,1,-1,1,-1,0,1,-1,1,-1]\]
with quadratic norm $-\frac{7}{6}$ generates the cyclic group $\mathbb Z/(12)$. They satisfy also the relation $B_4. B_{12}=\frac{1}{4}$.

We are going to prove that $T(Z_{-6})=[8 \quad 4 \quad 8]$.

Let \[ A=\begin{pmatrix}                                                                                                                                                                          
  8 & 4\\                                                                                                                                                                                         
  4 & 8\\                                                                                                                                                                                         
\end{pmatrix}.
\]
Applying lemma \ref{lem:2.1} we get
 \[ D=\begin{pmatrix}                                                                                                                                                                             
  4 & 0\\                                                                                                                                                                                         
  0 & 12\\                                                                                                                                                                                        
\end{pmatrix} \quad   U=\begin{pmatrix}                                                                                                                                                           
  0 & 1\\                                                                                                                                                                                         
  1 & -2\\                                                                                                                                                                                        
\end{pmatrix} \quad  V=\begin{pmatrix}                                                                                                                                                            
  1 & 2\\                                                                                                                                                                                         
  0 & -1\\                                                                                                                                                                                        
\end{pmatrix}.                                                                                                                                                                                     
\]
We deduce $G_A=\mathbb Z/(4) \oplus \mathbb Z/(12)$ and get a vector $b$ with quadratic norm $1/2$ generating $\mathbb Z/(4)$, a vector $c$ with quadratic norm $1/6$ generating $\mathbb Z/(12)$ such that $b.c=-1/4$.
It follows that the vector $2b+c$ is another generator of the cyclic group $\mathbb Z/(12)$ with quadratic norm $7/6$ and $b.(2b+c)\equiv -1/4  (\text{mod. } \mathbb Z)$, hence

\[ T(Z_{-6})=[8 \quad 4 \quad 8].\]

\subsection{$T(Z_{4})$}
A Weierstrass equation is given by
\[(F_{4})\qquad y^2=x^3+x^2(-8t^2+256t+2048)-16xt^4(t+7)\]
with singular fibers $I_8(0)$, $I_3(-8)$, $I_2(-7)$, $III^*(\infty)$.
The rank is $1$ and the infinite section given by the point $P=[-256(t+7),64(t+7)(t^2-64)]$ generates the Mordell-Weil group. We obtain that this infinite section does not cut the zero section and cuts only $\eta_1$, $\beta_1$ and $\alpha_1$.

We deduce the Gram matrix of the N\'eron-Severi lattice

\[NS= \left (
\begin{smallmatrix}
	-2& 1 & 0 & 0 & 0 & 0 & 0 & 0 & 0 & 0 & 0 & 0 & 0 & 0 & 0 & 0 & 0 & 0 & 0& 0\\
	1 & 0 & 1 & 0 & 0 & 0 & 0 & 0 & 0 & 0 & 0 & 0 & 0 & 0 & 0 & 0 & 0 & 0 & 0&1\\
	0 & 1 & -2 & 0 & 0 & 1 & 0 & 0 & 0 & 1 & 0 & 0 & 0 & 0 & 0 & 0 & 0 & 0 & 1&0\\
	0& 0 & 0 & -2 & 1 & 0 & 0 & 0 & 0 & 0 & 0 & 0 & 0 & 0 & 0 & 0 & 0 & 0 & 0&0\\
	0 & 0 & 0 & 1 & -2 & 1 & 0 & 0 & 0 & 0 & 0 & 0 & 0 & 0 & 0 & 0 & 0 & 0 & 0&0\\
	0& 0 & 1 & 0 & 1 & -2 & 1 & 0 & 0 & 0 & 0 & 0 & 0 & 0 & 0 & 0 & 0 & 0 & 0&0\\
	 0 & 0 & 0 & 0 & 0 & 1 & -2 & 1 & 0 & 0 & 0 & 0 & 0 & 0 & 0 & 0 & 0 & 0 & 0&0\\
	0& 0 & 0 & 0 & 0 & 0 & 1 & -2 & 1 & 0 & 0 & 0 & 0 & 0 & 0 & 0 & 0 & 0 & 0 &0\\
	0 & 0 & 0 & 0 & 0 & 0 & 0 & 1 & -2 & 0 & 0 & 0 & 0 & 0 & 0 & 1 & 0 & 0 & 0 & 0\\
	0 & 0 & 1 & 0 & 0 & 0 & 0 & 0 & 0 & -2 & 1 & 0 & 0 & 0 & 0 & 0 & 0 & 0 & 0 & 1\\
	0 & 0 & 0 & 0 & 0 & 0 & 0 & 0 & 0 & 1 & -2 & 1 & 0 & 0 & 0 & 0 & 0 & 0 & 0 & 0\\
	0 & 0 & 0 & 0 & 0 & 0 & 0 & 0 & 0 & 0 & 1 & -2 & 1 & 0 & 0 & 0& 0 & 0 & 0 & 0\\
	0 & 0 & 0 & 0 & 0 & 0 & 0 & 0 & 0 & 0 & 0 & 1 & -2 & 1 & 0 & 1 & 0 & 0 & 0 & 0\\
	0 & 0 & 0 & 0 & 0 & 0 & 0 & 0 & 0 & 0 & 0 & 0 & 1 & -2 & 1 & 0 & 0 & 0 & 0 & 0\\
	0 & 0 & 0 & 0 & 0 & 0 & 0 & 0 & 0 & 0 & 0 & 0 & 0 & 1 & -2 & 0 & 0 & 0 & 0 & 0\\
	0 & 0 & 0 & 0 & 0 & 0 & 0 & 0 & 0 & 0 & 0 & 0 & 1 & 0 & 0 & -2 & 0 & 0 & 0 & 0\\
	0 & 0 & 0 & 0 & 0 & 0& 0 & 0 & 0 & 0 & 0 & 0 & 0 & 0 & 0 & 0 & -2 & 1 & 0 & 1\\
	0 & 0 & 0 & 0 & 0 & 0 & 0 & 0 & 0 & 0 & 0 & 0 & 0 & 0 & 0 & 0 & 1 & -2 & 0 & 0\\
	0 & 0 & 1 & 0 & 0 & 0 & 0 & 0 & 0 & 0 & 0 & 0 & 0 & 0 & 0 & 0 & 0 & 0 &
        -2 & 1\\
	0& 1 &  0 & 0 & 0 & 0 & 0 & 0 & 0 & 1 & 0 & 0 & 0 & 0 & 0 & 0 & 1 & 0 & 1 & -2
      \end{smallmatrix}
    \right )
    \]

and its determinant $-32$. The elementary divisors are $2$ and $16$.  Applying lemma \ref{lem:2.1} we find
\[G_{NS}=\mathbb Z/(2) \oplus \mathbb Z/(16).\]
Moreover the vector
\[B_2=[0,0,-1,0,0,0,1,-1,1,0,0,0,0,0,0,0,-1,1,0,1]\]
with quadratic norm $-\frac{1}{2}$ generates the cyclic group $\mathbb Z/(2)$ and the vector
\[B_{16}=[-1,0,-1,0,0,0,-1,1,0,0,0,0,0,0,0,0,0,1,1,-2]\]
with quadratic norm $-\frac{1}{16}$ generates the cyclic group $\mathbb Z/(16)$.
They satisfy also the relation $B_2 . B_{16}=\frac{1}{2}$. Now the vector $B_2-B_{16}$ is another generator of the cyclic group $\mathbb Z/(16)$ since its quadratic norm is equal to $-\frac{25}{16}$ and it follows $B_2 . (B_2 -B_{16})=-1 \equiv 0 (\text{mod.} \mathbb Z)$.

Hence
\[T(Z_4)=[2\quad  0 \quad 16].\]

\subsection{$T(Z_{-12})$}
Its Weierstrass equation is given by

\[(F_{-12}) \quad y^2=x^3+(30t^2-448t+1536)x^2-t^4(t-9)x\]

with 2-torsion and singular fibers  $I_8(0)$, $I_3(8)$, $I_2(9)$ $III^*(\infty)$. Its rank is $1$ and
\[P=[-\frac{t^4(t+72)^2}{1728(t-24)^2}, \frac{\sqrt{-3}t^4(t+72)}{124416(t-24)^3}(t^4+144t^3-20736t^2+138240t+995328)]\]
of infinite order generates the Mordell-Weil group. Moreover this infinite section $(P)$ cuts the $2$-torsion section $s_2$, the zero section and $\theta_4$ and we deduce the Gram matrix of the N\'eron-Severi lattice of $Z_{-12}$.

\[NS= \left (
\begin{smallmatrix}
        -2& 1 & 0 & 0 & 0 & 0 & 0 & 0 & 0 & 0 & 0 & 0 & 0 & 0 & 0 & 0 & 0 & 0 & 0 & 1\\
        1 & 0 & 1 & 0 & 0 & 0 & 0 & 0 & 0 & 0 & 0 & 0 & 0 & 0 & 0 & 0 & 0 & 0 & 0 & 1\\
        0 & 1 & -2 & 0 & 0 & 1 & 0 & 0 & 0 & 1 & 0 & 0 & 0 & 0 & 0 & 0 & 0 & 0 & 1 & 1\\
        0& 0 & 0 & -2 & 1 & 0 & 0 & 0 & 0 & 0 & 0 & 0 & 0 & 0 & 0 & 0 & 0 & 0 & 0 & 0\\
        0 & 0 & 0 & 1 & -2 & 1 & 0 & 0 & 0 & 0 & 0 & 0 & 0 & 0 & 0 & 0 & 0 & 0 & 0 & 0\\
        0& 0 & 1 & 0 & 1 & -2 & 1 & 0 & 0 & 0 & 0 & 0 & 0 & 0 & 0 & 0 & 0 & 0 & 0 & 1\\
        0 & 0 & 0 & 0 & 0 & 1 & -2 & 1 & 0 & 0 & 0 & 0 & 0 & 0 & 0 & 0 & 0 & 0 & 0 & 0\\
        0& 0 & 0 & 0 & 0 & 0 & 1 & -2 & 1 & 0 & 0 & 0 & 0 & 0 & 0 & 0 & 0 & 0 & 0 & 0\\
        0 & 0 & 0 & 0 & 0 & 0 & 0 & 1 & -2 & 0 & 0 & 0 & 0 & 0 & 0 & 0 & 0 & 0 & 0 & 0\\
        0 & 0 & 1 & 0 & 0 & 0 & 0 & 0 & 0 & -2 & 1 & 0 & 0 & 0 & 0 & 0 & 0 & 0 & 0 & 0\\
        0 & 0 & 0 & 0 & 0 & 0 & 0 & 0 & 0 & 1 & -2 & 1 & 0 & 0 & 0 & 0 & 0 & 0 & 0 & 0\\
        0 & 0 & 0 & 0 & 0 & 0 & 0 & 0 & 0 & 0 & 1 & -2 & 1 & 0 & 0 & 0& 0 & 0 & 0 & 0\\
        0 & 0 & 0 & 0 & 0 & 0 & 0 & 0 & 0 & 0 & 0 & 1 & -2 & 1 & 0 & 1 & 0 & 0 & 0 & 0\\
        0 & 0 & 0 & 0 & 0 & 0 & 0 & 0 & 0 & 0 & 0 & 0 & 1 & -2 & 1 & 0 & 0 & 0 & 0 & 0\\
        0 & 0 & 0 & 0 & 0 & 0 & 0 & 0 & 0 & 0 & 0 & 0 & 0 & 1 & -2 & 0 & 0 & 0 & 0 & 0\\
        0 & 0 & 0 & 0 & 0 & 0 & 0 & 0 & 0 & 0 & 0 & 0 & 1 & 0 & 0 & -2 & 0 & 0 & 0 & 0\\
        0 & 0 & 0 & 0 & 0 & 0& 0 & 0 & 0 & 0 & 0 & 0 & 0 & 0 & 0 & 0 & -2 & 1 & 0 & 0\\
        0 & 0 & 0 & 0 & 0 & 0 & 0 & 0 & 0 & 0 & 0 & 0 & 0 & 0 & 0 & 0 & 1 & -2 & 0 & 0\\
        0 & 0 & 1 & 0 & 0 & 0 & 0 & 0 & 0 & 0 & 0 & 0 & 0 & 0 & 0 & 0 & 0 & 0 & -2 & 0\\
        1 & 1 & 1 & 0 & 0 & 1 & 0 & 0 & 0 & 0 & 0 & 0 & 0 & 0 & 0 & 0 & 0 & 0 & 0 & -2\\
      \end{smallmatrix}
    \right )
    \]

The determinant is $-96$. Applying lemma \ref{lem:2.1} we find
\[G_L\simeq \mathbb Z/(2) \oplus \mathbb Z/(48).\]
Moreover the vector
\[B_{2}=[2,-9,11,0,0,-7,-2,0,0,-11,11,-9,-2,2,9,0,3,-3,0,-2]\]
with quadratic norm $-\frac{1}{2}$ generates the cyclic group $\mathbb Z/(2)$ and the vector

\[B_{48}=[1,-3,4,0,0,-2,-1,0,0,-4,4,-3,-1,1,3,0,1,-1,0,-1]\]

with quadratic norm $-\frac{29}{48}$ generates the cyclic group $\mathbb Z/(48)$. They satisfy also the relation

\[B_2.B_{48}=-171/2 \equiv -1/2 (\text{mod.} \mathbb Z).\]
Let \[ A=\begin{pmatrix}                                                                                                                                                                          
  10 & 2\\                                                                                                                                                                                        
  2 & 10\\                                                                                                                                                                                        
\end{pmatrix}\]
with Smith form $[2 \quad 0 \quad 48]$. Applying lemma \ref{lem:2.1} we find $G_A \simeq \mathbb Z/(2) \oplus \mathbb Z/(48)$,  a vector $u$ with quadratic norm $1/2$ generating the cyclic group $\mathbb Z/(2)$, a vector $v$ with quadratic norm $5/48$ generating the cyclic group $\mathbb Z/(48)$ and satisfying $u.v=-1/2$. Since the vector $5v$ with quadratic norm norm $29/48$ is another generator of the cyclic group $\mathbb Z/(48)$  satisfying $u.5v=-5/2 \equiv 1/2 (\text{mod.} \mathbb Z)$ we deduce

\[ T(Z_{-12})=[10 \quad 2 \quad 10].\]

\section{Jacobian varieties of the singular $K3$ surfaces $Z_k$}
Denote $J_k$ the Jacobian variety of $Z_k$. Sometimes it is easy to get the transcendental lattice of $J_k$ from a Weierstrass equation. For the remaining cases we use the following Keum's results \cite{K} Lemma $2.1$ and Corollary $2.2$. More precisely, if $S$ (resp. $J(S)$) denotes the $K3$ surface (resp. its Jacobian surface) we use the formula
\[\det(Pic(S))=l^2 \det(Pic(J(S))\]
and its corollary: if $\det(Pic(S))$ is square free, then every elliptic fibration on $S$ is a Jacobian fibration i.e. has a section.

By specialization of the Weierstrass equation of $J_k$ we get
\[ (J_{12}) \qquad y^2=x^3+144t^2x^2+3072t^3(t-1)x+16384t^5(t-1)^2\]
with singular fibers $III^*(0)$, $I_3(1)$, $I_2(-1)$, $II^*(\infty)$. The rank is $0$ and by Shimada and Zhang \cite{SZ}, it follows
\[T(J_{12})=[2 \quad 0 \quad 6].\]
Also
\[(J_{-6}) \qquad y^2=x^3+36t^2x^2+192t^3(t-1)x+256t^5(t-1)^2\]
with singular fibers $III^*(0)$, $I_3(1)$, $I_2(-4)$, $II^*(\infty)$. The rank is $0$ and by Shimada and Zhang \cite{SZ}, it follows
\[T(J_{-6})=[2 \quad 0 \quad 6].\]
Moreover
\[(J_4) \qquad y^2=x^3+16t^2x^2+512t^3(t-1)x+4096t^5(t-1)^2\]
with singular fibers $III^*(0)$, $I_4(1)$, $I_1(\frac{32}{27})$, $II^*(\infty)$. The rank is $0$ and by Shimada and Zhang \cite{SZ}, it follows
\[T(J_{4})=[2 \quad 0 \quad 4].\]
We get also
\[(J_0) \qquad y^2=x^3+1024t^5(t-1)^2\]
with singular fibers $II^*(0)$, $IV(1)$, $II^*(\infty)$. The rank is $0$ and by Shimada and Zhang \cite{SZ} or by Utsumi \cite{U}, it follows
\[T(J_{0})=[2 \quad 1 \quad 2].\]

Since $T(Z_{-3})=[4 \quad 1 \quad 4]$, it follows $\det(Pic(Z_{-3}))=15$ is square free and every elliptic fibration on $Z_{-3}$ has a section. Hence
\[Z_{-3}=J_{-3}.\]
Indeed the genus $1$ fibration of $Z_{-3}$ leading to $J_{-3}$

\[W^2=-4mt^4-15m^2t^2+m^2(4m-10)t-3m^2\]
has the section $(t=0,y=m\sqrt{-3})$.

Since $\det(Pic(Z_{-12}))=-24\times 4$, according to the above formula, either $\det(Pic(J_{-12}))=-24 \times 4$ (impossible according all the possible discriminants of singular $K3$ surfaces of Verrill's family), or $\det(Pic(J_{-12}))=-24$ and
\[T(J_{-12})=[4 \quad 0 \quad 6].\]
Finally we conjecture from the fact that the discriminant of $J_{60}$ must be divisible by $5$
\[ T(J_{-36}) \overset{?}{=}[6 \quad 0 \quad 8] \qquad T(J_{60}) \overset{?}{=}[6 \quad 0 \quad 10].\]

\section{Final remarks}
We can now complete the table
\begin{center} 
 \begin{tabular}{|l|l|l|l|l|l|l|l|}
 \hline
 $Y_0$ & $[4 \quad 2 \quad 4]$ & $Z_{-36}$ &$?12a^2$  & $X_{-36}$ & $[2 \quad 0 \quad 6 ]$& $J_{-36}$ & $?[6 \quad 0 \quad 8]$\\
 
 $Y_2$ & $[2 \quad 0 \quad 4]$ & $Z_{-12}$ & $[10 \quad 2 \quad 10]$ & $X_{-12}$ & $[4 \quad 0 \quad 6 ]$ & $J_{-12}$ & $[4 \quad 0 \quad 6]$\\
 
  $Y_3$ & $[2 \quad 1 \quad 8]$ & $Z_{-6}$ & $[8 \quad 4 \quad 8]$ & $X_{-6}$ & $[6 \quad 0 \quad 8 ]$&$J_{-6}$ & $[2 \quad 0 \quad 6]$\\
 
 $Y_6$ & $[2 \quad 0 \quad 12]$ & $Z_{-3}$ & $[4 \quad 1 \quad 4]$ & $X_{-3}$ & $[6 \quad 0 \quad 10 ]$ & $J_{-3}$ & $[4 \quad 1 \quad 4]$\\  

 $Y_{10}$ & $[6 \quad 0 \quad 12]$ & $Z_{0}$ &$[2 \quad 0 \quad 6]$  & $X_{0}$ & $[2 \quad 0 \quad 6 ]$ & $J_0$ & $[2 \quad 1 \quad 2]$\\ 
 
   $Y_{18}$ & $[10 \quad 0 \quad 12]$ & $Z_{4}$ & $[2 \quad 0 \quad 16]$ & $X_{4}$ & $[2 \quad 0 \quad 4 ]$ & $J_4$ & $[2 \quad 0 \quad 4]$\\
  
   $Y_{102}$ & $[12 \quad 0 \quad 26]$ & $Z_{12}$ & $[2 \quad 0 \quad 24]$  & $X_{12}$ & $[2 \quad 1 \quad 2 ]$ & $J_{12}$ & $[2 \quad 0 \quad 6]$\\
   
    $Y_{198}$ & $[12 \quad 0 \quad 34]$ & $Z_{60}$ &$?15a^2$  & $X_{60}$ & $[4 \quad 1 \quad 4 ]$& $J_{60}$ & $?[6 \quad 0 \quad 10]$\\
    
     $Y_{k^2=20}$ & $[2 \quad 0 \quad 10]$ & &  &  &  & &\\
     
 $Y_{k^2=54}$ & $[4 \quad 0 \quad 12]$ & &  &  & & &\\
 
 $Y_{k^2=-12}$ & $[6 \quad 0 \quad 6]$ & &  &  & & & \\
 
 $Y_{k^2=-45}$ & $[8 \quad 2 \quad 8]$ & &  &  & & &\\
 \hline

\end{tabular}
\end{center}

\begin{itemize}
\item The Kummer surface $K_k$ associated to $Y_k$ has discriminant $24\times 4$ and we found in \cite{BL} that its Jacobian $K3$ surface has transcendental lattice $\langle -2 \rangle \oplus \langle 6 \rangle \oplus \langle 2 \rangle$ of discriminant $24$.
\item The $K3$ surface $Z_k$ with discriminant $24$ has for Jacobian variety $X_k$ with discriminant $6$.

  These two results are illustrations of the previous Keum's result \cite{K}.

\item Using Keum's corollary \cite{K}, since the discriminant of $X_k$ is $6$ without square factors, we can deduce that every elliptic fibration on $X_k$ has a section.
\item We notice also that $Y_{k^2=-45}$ is the Kummer surface of $Z_{-3}$ and that $Y_{k^2=54}$ is the Kummer surface of $X_0$ or of $X_{-36}$.
\item The $K3$ surface $Z_{-6}$ is the Kummer surface of $Y_0$ which is in turn the Kummer surface of $J_0$. Also $Y_0$ is the Kummer surface of $X_{12}$.
  
\item The $K3$ surface $Z_0$ is equal to $X_{-36}$ and $Y_2$ is equal to $X_4$.

\end{itemize}

Some of these remarks are useful to compute the $L$-series of the $K3$ surface.

\end{document}